\documentclass{amsart}
\pdfoutput=1
\usepackage[top=3cm, bottom=3cm, left=3cm, right=3cm]{geometry} 
\usepackage{graphicx}
\usepackage{color}
\usepackage{amssymb}
\DeclareMathOperator{\MCG}{\mathcal{MCG}}
\DeclareMathOperator{\CMCG}{\mathcal{CMCG}} 
\DeclareMathOperator{\Isom}{\emph{Isom}}

\newcommand{\PP}{\mathcal{P}}
\newcommand{\QQ}{\mathcal{Q}}
\newcommand{\DD}{\mathcal{D}}
\newcommand{\EE}{\mathcal{E}}

\newtheorem{theorem}{Theorem}[section]
\newtheorem{lemma}[theorem]{Lemma}

\theoremstyle{definition}
\newtheorem{definition}[theorem]{Definition}

\theoremstyle{cor}
\newtheorem{cor}[theorem]{Corollary}
\newtheorem*{question}{Question}
\newtheorem*{ack}{Acknowledgments}

\theoremstyle{remark}
\newtheorem{remark}[theorem]{Remark}

\numberwithin{equation}{section}

\begin{document}

\title{High Distance Knots in Closed 3-Manifolds}

\author{Marion Moore Campisi}
\address{Marion Moore Campisi, Department of Mathematics, University of Texas, Austin, 78712}
\email{marion@math.utexas.edu}

\author{Matt Rathbun}
\address{Matt Rathbun, Department of Mathematics, Michigan State University, 48824}
\email{mrathbun@math.msu.edu}

\keywords{distance, high distance, Heegaard splittings, knots, pants decompositions, curve complex, coarse geometry, disk set, coarse mapping class group}

\begin{abstract}
Let $M$ be a closed $3$-manifold with a given Heegaard splitting. We show that after a single stabilization, some core of the stabilized splitting has arbitrarily high distance with respect to the splitting surface. This generalizes a result of Minsky, Moriah, and Schleimer for knots in $S^3$. We also show that in the complex of curves, handlebody sets are either coarsely distinct or identical. We define the \emph{coarse mapping class group of a Heegaard splitting}, and show that if $(S, V, W)$ is a Heegaard splitting of genus $\geq 2$, then $\CMCG(S,V,W) \cong \MCG(S, V,W)$.
\end{abstract}

 \maketitle
 
\begin{center}
\today
\end{center}

\section{Introduction}
\label{section:introduction}
Hempel \cite{Hem3MAVFCC} developed the notion of the distance of a Heegaard splitting, generalizing the idea of a strongly irreducible Heegaard splitting (distance $\geq 2$). The notion of distance of Heegaard splittings has become increasingly important. Thompson introduced the Disjoint Curve Property for Heegaard splittings \cite{ThoDCPG2M}, which explores the case distance $= 2$. She uses this property to find a necessary condition for a genus $2$ manifold to be hyperbolic. Schleimer showed that all sufficiently high genus splittings of closed manifolds have the property \cite{SchlDCP}, and later Kobayashi and Rieck improved the genus bound to something linear in the number of tetrahedra in a triangulation of the manifold \cite{KobRieLBGHSWDGT2}. Lustig and Moriah have made use of a combinatorial criterion to ensure splittings are high distance, and proved the existence of knots in $S^3$ with infinitely many Dehn surgeries yielding high distance splittings \cite{LusMorHDHSVFTT}. Work of Scharlemann-Tomova \cite{SchaTomAHGBD}, \cite{TomMBSRKD}, and Hartshorne \cite{HarHSHMHB} have linked distance to the existence of incompressible surfaces, bridge surfaces, and even other Heegaard splittings. 
 
Minsky, Moriah, and Schleimer \cite{MinMorSchHDK} observe that construction of knots whose exteriors have high distance is relatively easy, first constructing a splitting of the desired distance \cite{Hem3MAVFCC}, \cite{KobNisNSC3MHGGHS}, and then removing a knot. The construction of these splittings, however, gives one little control of the resulting $3$-manifold. If we wish to specify the $3$-manifold, the question becomes more subtle. They prove the existence of knots in $S^3$ with arbitrarily high distance splittings of their exteriors. They also ask the question of whether a similar property holds in any $3$-manifold.
 
 We answer this question affirmatively and prove:
 
 \begin{theorem}
For any closed $3$-manifold $M$, any $g \geq genus(M) + 1$, and any $n > 0$, there is a knot $K \subset M$ and a genus $g$ splitting of $M - n(K)$ having distance greater than $n$.
\end{theorem}

In fact, we prove that one can start with any Heegaard splitting of $M$, and modify it slightly to be the splitting in question:

\begin{theorem}
\label{theorem:Main Theorem}
Given a closed $3$-manifold $M$ with a Heegaard splitting $(S, V, W)$, and any $n > 0$, $(S, V, W)$ can be stabilized once to give $(S', V', W')$ such that there exists a core $K$ of $(S', V', W')$ with $d((V' - n(K)), W') > n$.
\end{theorem}

In 2001, Abrams and Masur posed the following question:

\begin{question}
If $\DD_V$ and $\DD_W$ are two disk complexes having bounded Hausdorff distance in the curve complex, then are the disk complexes identical?
\end{question}

A slight modification of the techniques of the proof of Theorem \ref{theorem:Main Theorem} answers this question affirmatively. In Theorem \ref{theorem:CoarselyDistinct}, we show that handlebody disk sets are either identical, or coarsely distinct. In other words, a handlebody disk set is uniquely determined by its coarse geometry. Schleimer has also announced a proof of this, using different techniques.

Rafi and Schleimer \cite{RafSchlCCCBR} show that the mapping class group of a surface is isomorphic to the quasi-isometry group of the curve complex. As a corollary to Theorem \ref{theorem:CoarselyDistinct}, we prove an analogous result to Rafi's and Schleimer's. The so-called Coarse Mapping Class Group of a Heegaard splitting is isomorphic to the genuine Mapping Class Group of the splitting.

In section \ref{section:definitions} we recount definitions and background. In section \ref{section:graphs} we discuss some connections between Heegaard splittings and planar graphs and prove some useful lemmas concerning the construction of ``nice" pants decompositions for splittings. In section \ref{section:outline} we discuss the outline of the proof of Theorem \ref{theorem:Main Theorem}. In section \ref{sec:PantsandCurve} we prove the main lemmas used in the proof of Theorem \ref{theorem:Main Theorem}, and in section \ref{sec:proof} we prove the theorem itself. Finally, in section \ref{section:corollaries}, we recount some corollaries, discuss consequences for the complex of curves, and prove Theorem \ref{theorem:CoarselyDistinct}.

\vspace{3 mm}
\begin{ack} We would like to thank Abigail Thompson for her advice and helpful skepticism. We would also like to thank Jesse Johnson for the coarse geometric interpretation. 
The authors were supported in part by NSF VIGRE Grants DMS-0135345 and DMS-0636297.
\end{ack}

\section{Definitions}
\label{section:definitions}

\subsection{Compression Bodies}

\begin{definition} Let $S$ be the boundary of a $3$-manifold $M$, and let $D$ be a compressing disk (or possibly a collection of compression disks) for $S$. We fix some notation. Let $M | D = M \setminus (D \times (-\epsilon, \epsilon))$, and  $S| \partial D = S \setminus (\partial D \times (-\epsilon, \epsilon))$, for some small $\epsilon$. For convenience, we will also sometimes refer to this as $S | D$.

\end{definition}

\begin{definition} A \emph{compression body} $V$ is the result of taking the product of a surface with $[0, 1]$, and attaching $2$-handles along $S \times \{0\}$, and then $3$-handles along any resulting $2$-sphere components. $S \times \{1\}$ is called $\partial_+ V$, and $\partial V \setminus \partial_+ V$ is called $\partial_- V$. A \emph{handlebody} is a compression body where $\partial_- V = \emptyset$. A \emph{Heegaard splitting} is a triple $(S, V, W)$, where $S$ is a surface, $V$ and $W$ are compression bodies, and $\partial_+ V = \partial_+ W = S$.
\end{definition}

\begin{definition} A \emph{cut system} for a handlebody $V$ is a collection of disjoint, non-parallel compressing disks $\DD$ for $V$, such that $V | \DD$ is a single $3$-ball.
\end{definition}

\begin{definition} If $(S, V, W)$ is a Heegaard splitting of a closed manifold $M$, and $\DD$ and $\EE$ are cut systems for $V$ and $W$, respectively, then the triple $(S, \partial \DD, \partial \EE)$ is called a \emph{Heegaard diagram} for the splitting. Note that a Heegaard diagram depends on the splitting, and on the cut systems, but a diagram does determine the $3$-manifold $M$ completely.
\end{definition}

In short, a Heegaard diagram is a compact means of encoding the information about the gluing map between two handlebodies. One can start with the surface (thickened slightly), then attach $2$-handles along the curves of $\partial \DD$ on one side, and then attach a $3$-handle along the resulting sphere boundary component. This is the handlebody $V$. Attaching $2$-handles along $\partial \EE$ on the other side, and then attaching a $3$-handle, gives $W$, and hence all of $M$.

\begin{definition} We will say that a \emph{pants decomposition of a compression body $V$} is a maximal set of disjoint, non-parallel essential curves $\PP$ in $\partial V$ such that each curve $p$ in $\PP$ bounds a disk in $V$.  Note that this definition applies even if $V$ is genus one. Note also that if $V$ is a handlebody of genus greater than $1$, then a pants decomposition decomposes $V$ into solid pairs of pants. We will say that a \emph{pants decomposition of a Heegaard splitting} $(S, V, W)$ is a pair $(\PP, \QQ)$ where $\PP$ is a pants decomposition of $V$ and $\QQ$ is a pants decomposition of $W$.
\end{definition}

\subsection{Surfaces}

Let $S$ be a closed, orientable surface of genus $\geq 2$, and fix a hyperbolic structure on $S$. 

\begin{definition}
A \emph{(geodesic) lamination} is a compact subset of $S$ which is a union of complete, disjoint geodesics on $S$. A \emph{measured lamination} is a lamination, together with a transverse measure. We denote by $\mathcal{ML}(S)$ the set of all measured laminations on $S$. Note that $\mathcal{ML}(S) \subset \{ \mbox{functions} \ f: \mathcal{C}_S^0 \to \mathbb{R}_{\geq 0}\}$, and therefore inherits a natural topology. Then define the set of \emph{projective measured laminations} of $S$ to be $\mathcal{PML}(S) = \mathcal{ML}(S) / \mbox{(scalar multiplication of measures)}$, with the quotient topology.
 \end{definition}
 
 $\mathcal{PML}(S)$ is homeomorphic to $S^{6g - 7}$, and can be viewed as a compactification of the Teichm\"{u}ller space of $S$. 
 
 \begin{definition}
A homeomorphism $\phi : S \to S$ is \emph{pseudo-Anosov} if there exists a $k > 1$, and a pair of transverse measured laminations $\lambda^{\pm}$, such that $\phi(\lambda^+) = k\lambda^+$, and $\phi(\lambda^-) = \frac{1}{k}\lambda^-$. $\lambda^+$ is called the \emph{stable lamination of $\phi$} and $\lambda^-$ is the \emph{unstable lamination}.
 \end{definition} 
 
(For more on laminations, see \cite{BonGLS} or \cite{CasBleASANT}.)

\begin{definition}
Let $P$ be a pair of pants. A \emph{seam} of $P$ is an essential properly embedded arc in $P$ which connects two distinct boundary components of $P$. A \emph{wave} of $P$ is an essentially properly embedded arc in $P$ with endpoints on the same boundary component of $P$. Suppose $P \subset S$ and $\alpha$ is a simple closed curve in $S$. We say $\alpha$ \emph{traverses a seam in $P$} (resp. \emph{traverses a wave in $P$}) if it intersects $P$ minimally (up to isotopy) and a component of $\alpha \cap P$ is a seam (resp. wave). If $\PP$ is a pants decomposition of a compression body $V$, then a curve $\alpha$ on $\partial V$ is said to \emph{traverse all the seams of $\PP$} if $\alpha$ traverses every seam of every pair of pants of $\partial V | \PP$.
\end{definition}

\begin{definition} (see \cite{KobHSLPAH})
Let $L \subset S$ be a geodesic lamination. Given a pants decomposition $\QQ$ we say that $L$ is \emph{full type with respect to $\QQ$} if $L$ traverses all the seams of $\QQ$.  
We will say that a pants decomposition $(\PP, \QQ)$ of a Heegaard splitting $(S, V, W)$ is \emph{full type} if every seam of $\QQ$ is traversed by a curve of $\PP$, i.e. $\PP$ is full type with respect to $\QQ$.  Note that this is not symmetric.  
\end{definition}

\begin{definition}
A segment $\sigma$ of a curve $\alpha$ with respect to two curves $\beta_1, \beta_2$ is a subinterval of $\alpha$ with $\partial \sigma = (\sigma \cap \beta_1) \cup (\sigma \cap \beta_2)$. We denote a segment $\sigma(\alpha, \beta_1, \beta_2)$.
\end{definition}

\begin{remark} If $P$ has boundary components $p_1, p_2,$ and $p_3$, then $\alpha$ traverses a wave of $P$ if it intersects $P$ minimally and has a segment $\sigma(\alpha, p_i, p_j)$ for some $i = j$.
\end{remark}

\subsection{Graphs}

\begin{definition}
A \emph{(undirected) graph} is a pair $G = (X, E)$ of sets, with $E\subset X\times X$, such that if $e = (x, y) \in E$, then $\overline{e} = (y, x) \in E$. The elements of $X$ are the \emph{vertices} of the graph, and the elements of $E$ are called \emph{edges}. If $e = (x, y) \in E$, then we say that $e$ is an edge between $x$ and $y$, or that $e$ is \emph{incident} to $x$ and $y$, allowing for the possibility that $x=y$. 
\end{definition}

\begin{definition}
Let $G = (X, E)$ be a graph, let $x \in X$, and let $e = (y, z) \in E$. $G - x$ is the graph which is the result of removing $x$ from $X$, and removing any edges incident to $x$ from $E$. We call this \emph{removing a vertex}. $G - e$ is the graph $(X, E \setminus \{e, \overline{e}\})$. We call this \emph{deleting an edge}.  $G/e$ is the graph which is the result of removing $e$ and $\overline{e}$ from $E$, and of identifying $y$ with $z$. We call this \emph{contracting an edge}. 
\end{definition}

\subsection{Curve Complex}

\begin{definition} Let $S$ be a closed, orientable surface of genus $\geq 2$. Then the \emph{curve complex} of $S$ is the complex $\mathcal{C}_S$ whose vertices correspond to isotopy classes of essential, simple closed curves on $S$, and such that a collection of distinct vertices $v_0, \dots, v_n$ bound an $n$-simplex if representatives from each of the corresponding isotopy classes can be chosen to be simultaneously disjoint. In particular, we will be concerned with the $1$-skeleton, $\mathcal{C}_S^1$. The \emph{distance} between two vertices is the smallest number of edges in any path between the vertices in $\mathcal{C}_S^1$.
\end{definition}

Hempel \cite{Hem3MAVFCC} generalized the notions of reducibility/irreducibility and weak reducibility/strong irreducibility of Heegaard splittings by introducing the notion of distances of splittings. We will express the definition in the language of the curve complex. 

\begin{definition}
Let $(S, V, W)$ be a Heegaard splitting. Let $\DD_V$ be the subset of the curve complex corresponding to curves on $S$ which bound disks in $V$. Define $\DD_W$ similarly. These are called the \emph{disk sets of the handlebodies}. Then the \emph{(Hempel) distance} of the Heegaard splitting, denoted $d(V, W)$ or $d(S)$, is the distance in $\mathcal{C}_S$, $d(\DD_V, \DD_W)$.
\end{definition}

Johnson and Rubinstein generalized the notion of the mapping class group of a surface, and introduced the mapping class group of a Heegaard splitting:

\begin{definition} \cite{JohRubMCGHS}
Let $(S, V, W)$ be a Heegaard splitting of a manifold $M$. $Aut(S, V, W)$ is the set of automorphisms of $M$ that send $S$ to itself. The \emph{mapping class group of $(S, V, W)$}, $\MCG(S, V, W)$, is the group of the connected components of $Aut(S, V, W)$.
\end{definition}

As we are going to analyze $\mathcal{C}_S$ coarsely, we will define the following.

\begin{definition}\
\begin{itemize}
\item We say a map between two metric spaces, $f: (X, d_X) \to (Y, d_Y)$ is a \emph{quasi-isometric embedding} for some $k \geq 1$, $c \geq 0$, if for every $x_1, x_2 \in X$,
$$\frac{1}{k} \left( d_X(x_1, x_2) - c \right) \leq d_Y(f(x_1), f(x_2)) \leq kd_X(x_1, x_2) + c.$$

\item We say that a quasi-isometric embedding, $f: (X, d_X) \to (Y, d_Y)$, is \emph{$D$-dense} if there exists a $D > 0$ such that for every $y \in Y$, there is an $x \in X$ such that $d_Y(y, f(x)) < D$.

If $f: (X, d_X) \to (Y, d_Y)$ is a quasi-isometric embedding which is also $D$-dense, then we say that $f$ is a \emph{quasi-isometry} and we say that $(X, d_X)$ and $(Y, d_Y)$ are \emph{quasi-isometric}.
\end{itemize}
\end{definition} 

For $S$ a (closed) surface of genus $\geq 2$, we know \cite{IvaACCTS} every isometry of $\mathcal{C}_S$ is induced by a homeomorphism of $S$. In other words, $\Isom(\mathcal{C}_S) \cong \MCG(S)$. Looking at the mapping class group coarsely, we can replace isometries with quasi-isometries. So the following definition is justified.

\begin{definition} The \emph{coarse mapping class group} of a surface $S$, $\CMCG(S)$, is the group of quasi-isometries of $\mathcal{C}_S$, modulo the following relation. $f \sim g$ if there exists a bound $D$ such that for any $x \in \mathcal{C}_S$, $d(f(x), g(x)) \leq D$. This is also known as $QI(\mathcal{C}_S)$.
\end{definition}

The implicit generalization of Johnson and Rubinstein's mapping class group of a Heegaard splitting to coarse geometry is then the following:
 
\begin{definition}
The \emph{coarse mapping class group} of a splitting $(S, V, W)$, $\CMCG(S, V, W)$, is the group of quasi-isometries of $\mathcal{C}_S$, modulo $\sim$, with the additional condition that each quasi-isometry (class) move each handlebody disk set $\mathcal{D}_V$ and $\mathcal{D}_W$ only a bounded distance.
\end{definition} 

\section{Whitehead Graphs}
\label{section:graphs}

One can construct a graph from a Heegaard diagram as follows (see, for instance, \cite{StaWGH}). Begin with the diagram $(S, \partial \DD, \partial \EE)$. Let $\widehat{S} = S|\partial \EE$. Then $\widehat{S}$ is a $2g$-times punctured sphere, where $g = \mbox{genus}(S) = |\EE|$. Notice that $\partial \DD|\partial \EE$ is a collection of arcs properly embedded in $\widehat{S}$. We will call the \emph{Whitehead graph} the graph $G$ whose vertices correspond to the boundary components of $\widehat{S}$, and whose vertices are joined by an edge if there is a component of $\partial \DD |\partial \EE$ with one endpoint on each of the curves of $\partial \EE$ corresponding to these vertices. Considering $\partial{\widehat{S}} \cup \partial \DD|\partial \EE \subset \widehat{S}$, there is a natural projection map $\pi :  \partial{\widehat{S}} \cup (\partial \DD|\partial \EE) \to G$. Call $x_i^{\pm}$ the two vertices arising from the disk $E_i \in \EE$. (Note: since $\widehat{S}$ is a punctured sphere, $\partial \DD|\partial \EE$ is embedded in a planar surface, so the graph $G$ is planar. It is also worth noting that this process is symmetric. We could have chosen to consider $S|\partial \DD$, though this may have resulted in a different graph.) 

\begin{remark} Compare to Section 11.1 of \cite{AbrSchlDHS}. In this notation, we take $\mathcal{L}$ to be $\partial \DD$, and $C$ to be $\partial \EE$. Our Lemma \ref{lem:prettypants} was proven independently by Abrams and Schleimer in [Lemma 11.8, \cite{AbrSchlDHS}].
\end{remark}

It is reasonable to expect that there is some connection between the topology of the splitting, and the combinatorics of this graph. The connection is not as strong as one might hope, since multiple diagrams can correspond to the same splitting. However, there are useful connections to be found. We will use the following definition.

\begin{definition}
A graph $G$ is \emph{$2$-connected} if $|X| > 2$, and $G - Y$ is connected for every set $Y \subset X$ with $|Y| < 2$.
\end{definition}

In other words, the graph is connected, and there do not exist any vertices such that removing the vertex disconnects the graph.

There is a relationship between the connectivity of the graph, and the reducibility of the Heegaard splitting.

\begin{lemma} 
\label{lem:2conn} 
If a Heegaard splitting $(S, V, W)$ is irreducible, and a pair of cut systems for $V$ and $W$ intersect minimally, then the associated Whitehead graph is $2$-connected.
\end{lemma}

\begin{proof}
Choose disk sets $\DD$ and $\EE$ for $V$ and $W$ respectively so as to minimize the number of intersections between $\DD$ and $\EE$. If the Whitehead graph $G$ is not connected, then $\pi^{-1}(G)$ is not connected, and there exists a curve $\sigma \subset \widehat{S}$, disjoint from $\pi^{-1}(G)$, that separates two components of $\pi^{-1}(G)$. Since $\sigma$ separates two non-empty components of $\pi^{-1}(G)$, it is an essential curve in $S$. But it is disjoint from all boundary components and arcs of $\DD|\EE$. So $\sigma$ is an essential curve in $S = \partial W$, disjoint from all the curves $\EE$, and thus bounds a disk in $W$. But $\sigma$ is also an essential curve on $S = \partial V$, disjoint from all the curves $\DD$, and thus bounds a disk in $V$. This contradicts the assumption that $(S, V, W)$ was irreducible.

Now, if there were a vertex, say $x_i^+$, such that $G - x_i^+$ was disconnected, then there would exist a properly embedded arc $\sigma \subset \widehat{S}$, disjoint from $\pi^{-1}(G - x_i^+)$, with both endpoints on $x_i^+$, such that $\sigma$ separates $\widehat{S}$ into two non-empty components. Let $R$ be the component that does not contain $x_i^-$. Then consider re-identifying $\pi^{-1}(x_i^+)$ and $\pi^{-1}(x_i^-)$ along $E_i$. The result is a $2(|\EE| - 1)$-punctured torus. We can slide everything in $R$ along the $1$-handle corresponding to $E_i$ (which gets replaced by $E'$. The effect of this slide is to reduce the number intersections of $\DD$ with $E_i$, without changing any other intersections. But this contradicts our assumption of minimality of $\DD \cap \EE$.

\end{proof}
 
We have a useful lemma from Przytyscki :

\begin{lemma}[Przytyscki \cite{PrzGL}]
\label{lem:graphlemma}
Let $e$ be any edge of a $2$-connected graph $G$. Then

\begin{enumerate}
\item if $G$ has more than one edge, then $G - e$ and $G/e$ are connected. 
\item either $G/e$ or $G - e$ is $2$-connected. 
\end{enumerate}
\end{lemma}

\begin{cor}
\label{cor:graphcor}
If $G$ is a $2$-connected graph with more than two vertices, then there exists an edge $e$ such that $G/e$ is $2$-connected.
\end{cor}

\begin{proof}
Suppose there is no edge which can be collapsed while preserving the property of $2$-connectedness. Then select any edge, call it $e_1$, and remove it, calling the resulting graph $G_1$. $G_1$ is $2$-connected. If there were an edge $e$ of $G_1$ such that $G_1/e$ was $2$-connected, then $G/e$ would be $2$-connected. So remove another edge, say $e_2$. Again, $G_2$ is $2$-connected, with no collapsible edges, for a collapsible edge of $G_2$ would have been a collapsible edge of $G$. Continue this process until $G_n$ has only two edges. If $G_n$ has more than three vertices, there is an immediate contradiction, as the graph cannot be connected. If $G_n$ has exactly three vertices, then it is either disconnected, or it is a line graph, which is not $2$-connected. This is, again, a contradiction. Therefore, there must be an edge $e$ of $G$ such that $G/e$ is $2$-connected.
\end{proof}

Combining these results gives us the following:
\begin{lemma}
\label{lem:prettypants}
Let $(S, V, W)$ be an irreducible Heegaard splitting of a $3$-manifold. Then there exists a pants decomposition ($\PP, \QQ)$ of the splitting which is full type.
\end{lemma}

\begin{proof}
Let $\DD$ and $\EE$ be cut systems as in lemma \ref{lem:2conn}, and $G$ be the associated Whitehead graph. We will extend $\DD$ and $\EE$ into pants decompositions with the desired properties, proceeding inductively on the number of vertices in $G$.

If there are two vertices, then $\DD$ and $\EE$ each consist of a single disk, and $V$ and $W$ are solid tori. Then by virtue of being irreducible, $\DD$ intersects $\EE$, and by our convention, these are pants decompositions which satisfy the conclusion of the lemma. This is really a special case, and is not the base for our induction.

\begin{figure}[tb]
\begin{center}
\includegraphics[width=2in]{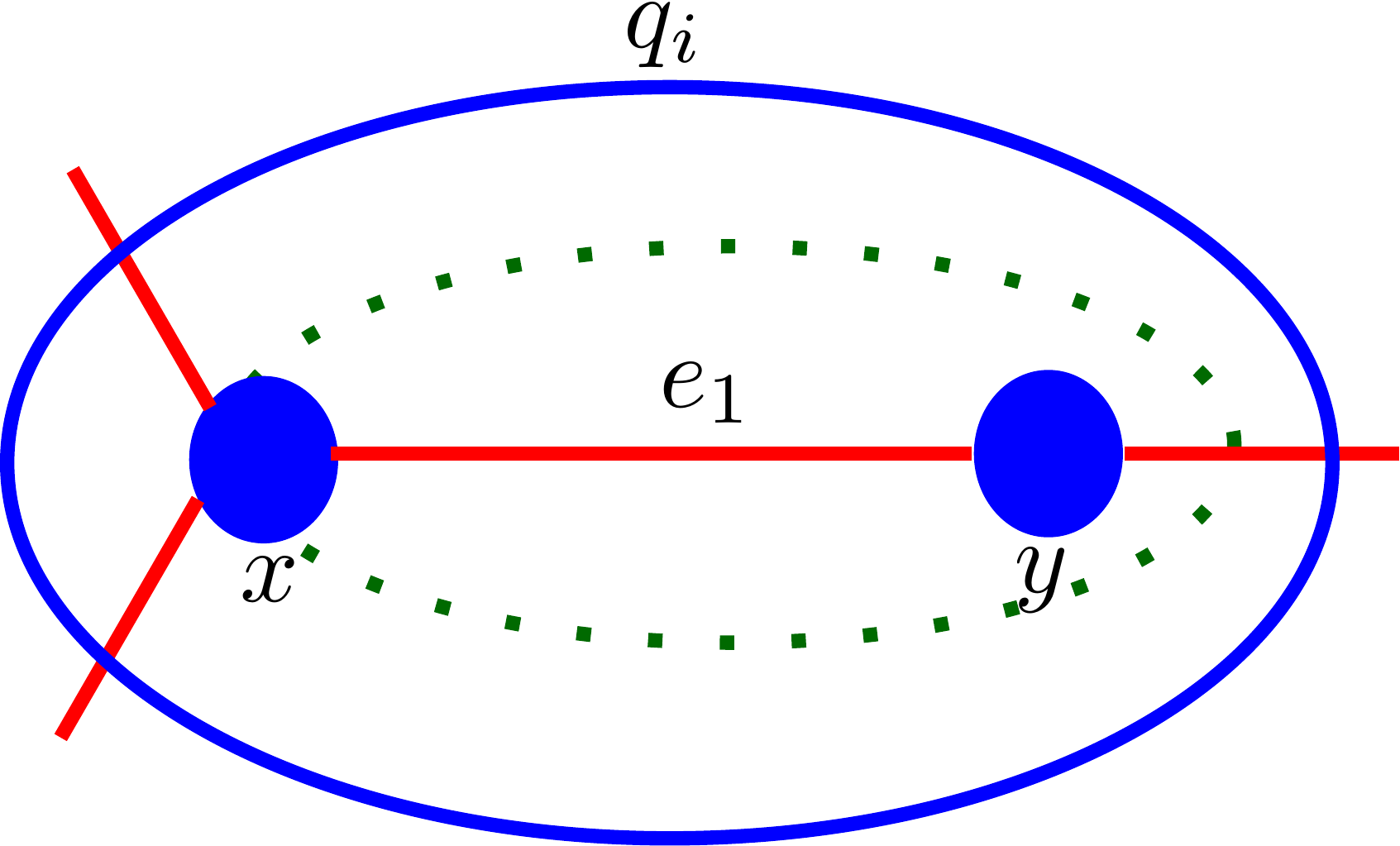}
\caption{Constructing $q_1 \in \QQ$.}
\label{figure:edge}
\end{center}
\end{figure}

Our base case is when $G$ has four vertices. By Corollary \ref{cor:graphcor}, there exists some edge, $e_1$, such that $G/e_1$ is $2$-connected. Suppose $e_1$ has endpoints $x$ and $y$. Pick an arc $\widetilde{e_1} \subset \widehat{S}$ of $\pi^{-1}(e_1)$. Let $q_1$ be a loop on $\widehat{S}$ that goes around $\pi^{-1}(x)$, $\pi^{-1}(y)$, and $\widetilde{e_1}$, cutting off a pair of pants $Q_1$ with boundaries corresponding to $\pi^{-1}(x)$, $\pi^{-1}(y)$, and $q_1$ (see Figure \ref{figure:edge}). That is, on the surface $S$, let $q_1$ be the band sum along $\widetilde{e_1}$ of the disks corresponding to $x$ and $y$. As the band sum of two disks, $q_1$ bounds a disk in $W$. $q_1$ will be a curve in our pants decomposition $\QQ$. 

Recall that edges of $G$ correspond to sub-arcs of the curves in $\DD$. Notice then that $\widetilde{e_1}$ traverses the seam of $Q_1$ between $\pi^{-1}(x)$ and $\pi^{-1}(y)$. Now, $q_1$ cannot separate the preimage of the graph, so at least one of the seams, say between $\pi^{-1}(x)$ and $q_1$ must be traversed by an arc of $\pi^{-1}(G)$. And finally, there must be an arc of $\pi^{-1}(G)$ which traverses the seam between $\pi^{-1}(y)$ and $q_1$, or else there would be an arc from $\pi^{-1}(x)$ to itself inside $Q_1$ which would contradict $2$-connectedness (see Figure \ref{figure:edge}). Thus, every seam of $Q_1$ is traversed by $\DD$.

Now, the effect on $\pi$ of collapsing the edge $e_1$ is to consider the entire pair of pants $Q_1$ to be mapped to a single vertex (see Figure \ref{figure:contract}). By lemma \ref{lem:graphlemma}, $G/e_1$ is $2$-connected, and has only $3$ vertices. But the only $2$-connected graph with $3$ vertices is a triangle. So the preimage of the remaining vertices and $q_1$ also form a pair of pants, and every seam is traversed by an arc of $\pi^{-1}(G/e_1)$, and thus by an arc of $\pi^{-1}(G)$.

\begin{figure}[tb]
\begin{center}
\includegraphics[width=2in]{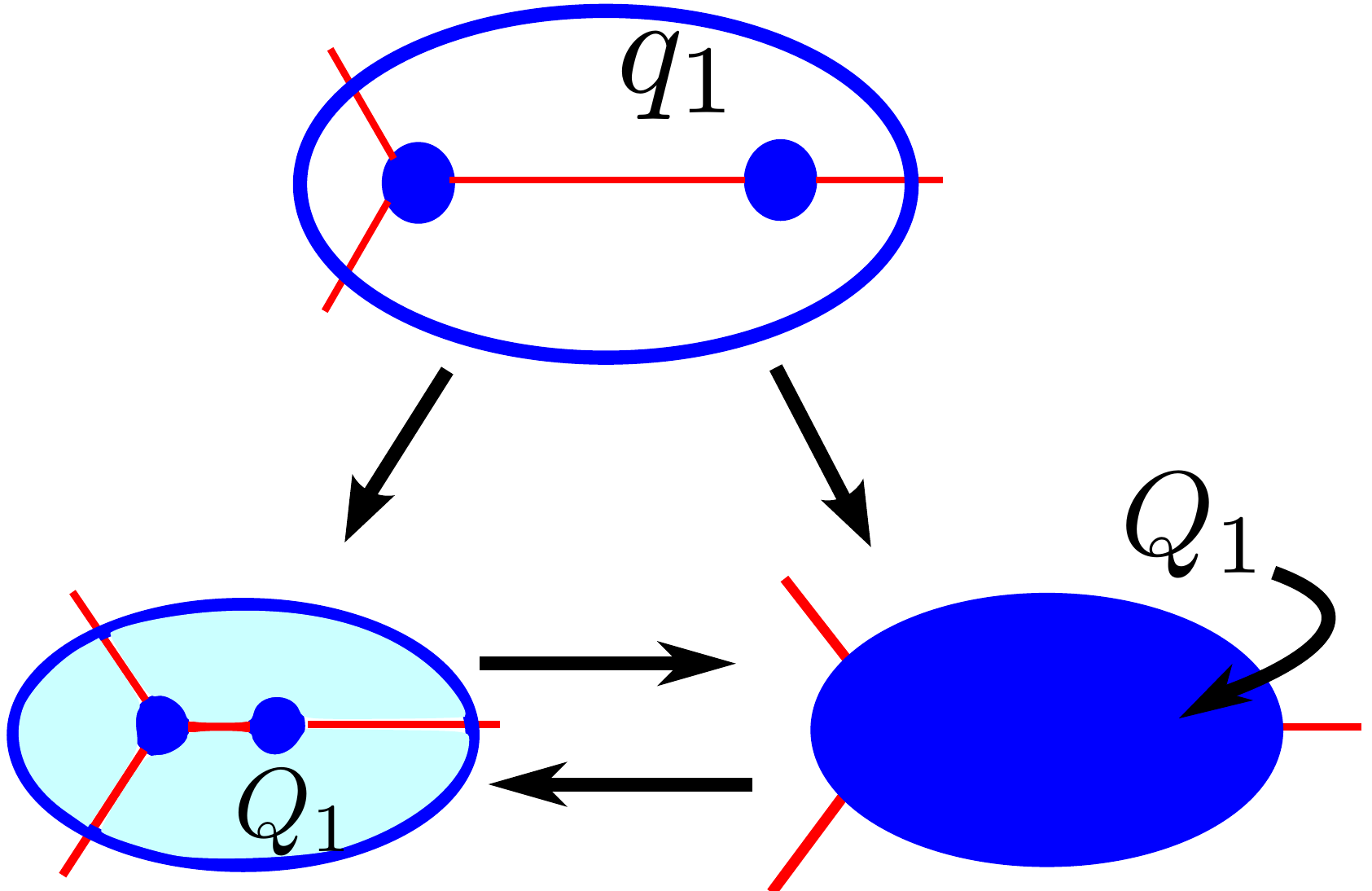}
\caption{Collapsing $e_1$.}
\label{figure:contract}
\end{center}
\end{figure}

Then, if $G$ has $2n$ vertices, similarly collapsing the graph along an edge guaranteed by Corollary \ref{cor:graphcor} reduces the number of vertices, and we conclude by induction. 

Thus, we can extend $\EE$ to a pants decomposition $\QQ$ of $W$ with the property that every seam of $\QQ$ is traversed by a curve of $\DD$. We then extend $\DD$ to a pants decomposition $\PP$ of $V$ in any way we like in order to finish the lemma.
\end{proof}

\section{Outline of Proof of Theorem \ref{theorem:Main Theorem}}
\label{section:outline}
In \cite{MinMorSchHDK}, Minsky, Moriah, and Schleimer prove that for any genus $g \geq 2$, and any distance $n$, there is a knot in $S^3$ and a genus $g$ Heegaard splitting of the knot complement with distance $> n$. Their method is to start with a standard Heegaard splitting $(S, V, W)$ for $S^3$, along with the standard pants decompositions for the two handlebodies. Then, they remove a core $K$ from $V$ which respects the pants decompositions, which changes the disk set for $V$ in a prescribed manner. They then use the pants decompositions to construct a particular pseudo-Anosov homeomorphism of $S$ which extends over $V$. Since the map extends over $V$, the image of the original $K$ is a new knot $K'$ in $S^3$. And since the map has particular properties, it pulls the disk sets for $(V \setminus n(K))$ and $W$ arbitrarily far apart under iteration, ensuring that $K'$ is a high distance knot.

The main lemma in their proof is the following:
\begin{lemma}[Lemma 2.1 of \cite{MinMorSchHDK}]
\label{lem:dis}
Suppose $X, Y \subset \mathcal{C}_S$, and let $\overline{X}, \overline{Y}$ denote their closures in $\mathcal{PML}(S)$. Let $\Phi$ be a pseudo-Anosov homeomorphism of $S$ with stable and unstable laminations $\lambda^{\pm}$. Assume that $\lambda^- \notin \overline{Y}$ and $\lambda^+ \notin \overline{X}$. Then $d(X, \Phi^n(Y)) \to \infty$ as $n \to \infty$.
\end{lemma}

They also make use of the following:
\begin{lemma}[Lemma 2.3 of \cite{MinMorSchHDK}]
\label{lem:lam}
No lamination in the closure of the disk set of a compression body traverses all the seams of a pants decomposition of that compression body.
\end{lemma}

The authors of \cite{MinMorSchHDK} use pants decompositions of $S^3$ which are concrete, so they can construct a pseudo-Anosov map whose stable lamination traverses all the seams of (is full type with respect to) both pants decompositions. Thus, by lemma \ref{lem:lam}, they conclude that the stable and unstable laminations of the pseudo-Anosov map are not in the closures of the disk sets of $V \setminus n(K)$ and $W$, so by Lemma \ref{lem:dis}, the map moves the disk sets further and further apart.

They point out that the same technique will not work in an arbitrary $3$-manifold. Consider the connected sum of several copies of $S^1 \times S^2$. The Heegaard splitting will have identical disk sets, so finding a pseudo-Anosov homeomorphism that extends to one of the handlebodies cannot increase distance, since the resulting Heegaard splitting will always be reducible. The remedy they suggest is to stabilize the Heegaard splitting once. This is the technique we use, for even a single stabilization provides enough asymmetry between the disk sets to construct the desired map. 

The heart of \cite{MinMorSchHDK} is actually constructing the pseudo-Anosov map in question. This is done using a train track on the Heegaard surface to construct a curve $a$ which 1) bounds a disk in $V$, and 2) is of full type with respect to both pants decompositions. Again, this is possible in $S^3$ because the pants decompositions are explicit. The difficulty for an arbitrary $3$-manifold is that there is no canonical choice of pants decompositions. We will overcome this obstacle in the following way:

\subsubsection*{Warm-Up} Using Lemma \ref{lem:prettypants}, we can construct pants decompositions $(\PP_i$, $\QQ_i)$ of full type for an \emph{irreducible} Heegaard splitting, along with a curve $\gamma_i$ which is of full type with respect to $\PP_i$ and $\QQ_i$. 

\subsubsection*{General Case} Given any arbitrary Heegaard splitting of a closed $3$-manifold, we will decompose the splitting into irreducible summands. We will then use the pants decompositions $\PP_i, \QQ_i$, and curves $\gamma_i$ on each irreducible summand provided by the Warm-Up to construct new pants decompositions $\PP$ and $\QQ$, and a single curve $\gamma$ on the original Heegaard surface which traverses all the seams of those pants decompositions. (Note: $(\PP$, $\QQ)$ will not be of full type, but $\gamma$ will be full type with respect to both $\PP$ and $\QQ$.)

\vspace{.25cm}

Then, given any Heegaard splitting, we can use these pants decompositions and the curve $\gamma$ in the construction of a pseudo-Anosov homeomorphism with the desired properties. We stabilize the splitting, and remove a core from one of the handlebodies. In the stabilized splitting, we will use the curve $\gamma$ to construct a meridian $a$ with the same properties as the meridian $a$ from \cite{MinMorSchHDK}, and proceed in the same fashion. (Notice that in the case $M = S^3$, we will recover the results of \cite{MinMorSchHDK} by way of a different curve.)

\section{Pants Decompositions and Seam-Traversing Curves}
\label{sec:PantsandCurve}
We already know from Lemma \ref{lem:prettypants} that there exist pants decompositions of full type for an irreducible Heegaard splitting. If a splitting is not irreducible however, we cannot hope to find pants decompositions that are so nice. All we need, however, are two pants decompositions and a \emph{curve} which is of full type with respect to both. We use the full type pants decomposition to find such curves on irreducible splittings. We then use the \emph{curves} on irreducible splittings to construct such full type \emph{curves} on reducible splittings.
 
\subsubsection*{Warm-Up} First we consider irreducible Heegaard splittings.

\begin{lemma}\label{lem:irreducible} Given any irreducible Heegaard splitting $(S_i, V_i, W_i)$ of a closed $3$-manifold, there exist pants decompositions $\PP_i$ and $\QQ_i$ for $V_i$ and $W_i$ respectively, and a curve $\gamma_i$ in $S_i$ that traversses every seam of both pants decompositions.
\end{lemma}

\begin{proof}
For $(S_i,V_i,W_i)$, construct a curve $\gamma_i$ in the following way:  

\begin{enumerate}
\item  If $(S_i, V_i,W_i)$ is a genus 1 splitting given by a $(p,q)$ curve:  

Choose $\gamma_i$ to be a $(-q,p)$ curve, unless $p=0$, in which case choose, say, $(-1,2)$.

\item If $(S_i, V_i,W_i)$ has genus $ \geq 2$:

\begin{itemize}
\item Start with a minimal Heegaard diagram $(S_i, \partial\DD_i, \partial\EE_i)$ for $(S_i, V_i, W_i)$. Construct $\mathcal{Q}_i$ as in Lemma \ref{lem:prettypants}, and extend $\DD_i$ to the standard pants decomposition $\mathcal{P}_i$ of $V_i$ (see Figure \ref{figure:StandardPants}). By construction, every seam of $\mathcal{Q}_i$ will be traversed by the curves of $\mathcal{P}_i$.

\item Choose a curve $\gamma_i$ carried by the train tracks $\tau_{\mathcal{P}_i}$, shown in Figure \ref{figure:tau}, which has weights of at least $2$ on every branch. This curve $\gamma_i$ exists by the argument in [Section 3.1, \cite{MinMorSchHDK}]. Note that since $\gamma_i$ is carried by $\tau_{\mathcal{P}_i}$ with weights $\geq 2$ on every branch, $\tau_{\mathcal{P}_i}$ contains all curves of $\mathcal{P}_i$, and every seam of $\mathcal{Q}_i$ is traversed by the curves of $\mathcal{P}_i$, it follows that all the seams of $\mathcal{Q}_i$ are traversed by $\gamma_i$. Further, by construction of $\tau_{{\PP}_i}$, $\gamma_i$ also traverses all the seams of $\PP_i$ and $\gamma$ is of full type with respect to $(\PP, \QQ)$
\end{itemize}

\end{enumerate}

\begin{figure}[tb]
\begin{center}
\includegraphics[width=3in]{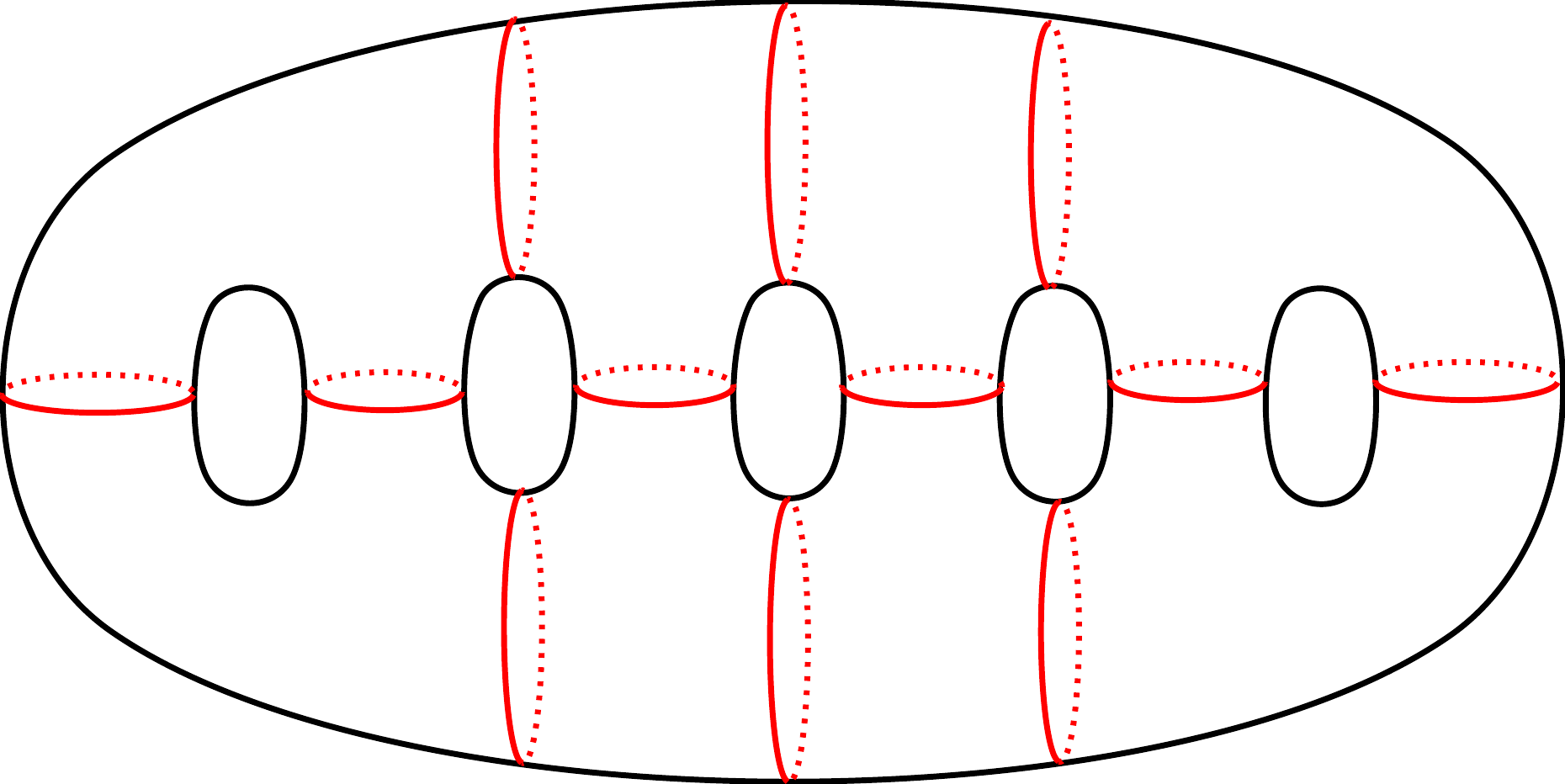}
\caption{The curves $\PP_i$.}
\label{figure:StandardPants}
\end{center}
\end{figure}

\begin{figure}[tb]
\begin{center}
\includegraphics[width=3in]{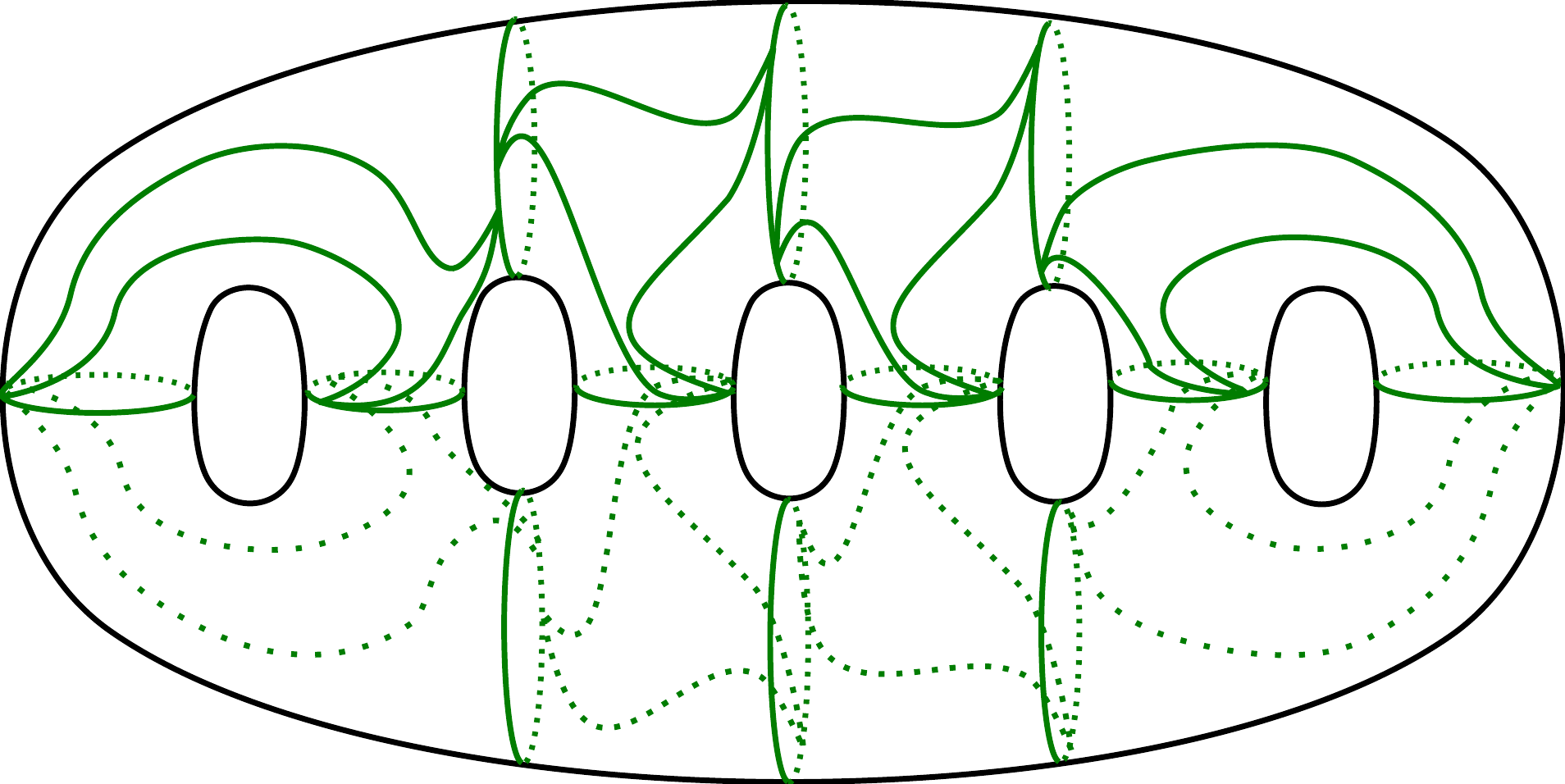}
\caption{The train track $\tau_{\PP_i}$.}
\label{figure:tau}
\end{center}
\end{figure}

\end{proof}

\subsubsection*{General Case} We then consider reducible Heegaard splittings.

 \begin{lemma}
\label{lem:curve}
Given any Heegaard splitting $(S, V, W)$ of a closed $3$-manifold, there exist pants decompositions $\PP$ and $\QQ$ for $V$ and $W$ respectively, and a curve $\gamma$ in $S$  of full type with respect to these pants deocompositions.
\end{lemma}

\begin{proof}
Let $(S, V, W)$ be a Heegaard splitting of a closed $3$-manifold $M$. If the splitting is reducible, then take a maximal collection of separating, reducing spheres $\{F_1, \dots, F_{l-1}\}$ for the splitting, so that $M = (S_1, V_1, W_1) \#_{F_1} (S_2, V_2, W_2) \#_{F_2} \dots \#_{F_{l-1}} (S_l, V_l, W_l)$, where $(S_i, V_i, W_i)$ are irreducible splittings of prime $3$-manifolds for $1 \leq i \leq k$, and are genus $1$ splittings of $S^3$ or $S^2 \times S^1$ for $k+1 \leq i \leq l$. By Lemma \ref{lem:irreducible} we can construct pants decompositions $\PP_i$ and $\QQ_i$, and a curve $\gamma_i$ on each summand $S_i$ individually. We now show that it is possible to patch them together to create our desired curve $\gamma$ on the entire surface $S$. We will do this by carefully taking the connect sum $(S,V,W)=(S_1,V_1, W_1) \# (S_2,V_2,W_2)\# \cdots \#(S_n,V_n,W_n)$ in such a way that we have control over the naturally induced pants decompositions.

For simplicity, we will suppose there are only two summands, as the general case is simply a repetition of this process. 

So for $j = 1,2$, we have $M_j = (S_j, V_j, W_j)$, and a curve $\gamma_j$ which is full type with respect to $(\PP_j, \QQ_j)$. 

Let $z_j$ be a point in $S_j|(\PP_j \cup \QQ_j \cup \gamma_j)$, in a component adjacent to a segment $\sigma = \sigma(\gamma_j, p_j^1, p_j^2)$ of $\gamma_j$ connecting two different curves $p_j^1$ and $p_j^2$ of $\PP_j$. Then $z_j$ is in the intersection of two pairs of pants, say $P_j$ and $Q_j$, with $p_j^1$ and $p_j^2$ being two boundary components of $P_j$. Let $\delta_j = \partial n(z_j)$. We are going to connect sum $S_1$ to $S_2$ by identifying $n(z_1)$ with $n(z_2)$. But we must extend $\PP_1 \cup \PP_2$ and $\QQ_1 \cup \QQ_2$ in order to obtain pants decompositions of the new surface $S = S_1 \# S_2$. 

To extend $\PP_j$, let $p^*_j$ be the unique curve in $P_j$ which is parallel to $p_j^1$ in $P_j$, but not in $P_j \setminus n(z_j)$. See Figure \ref{figure:connectsum}. (Notice that $p_j^*$ bounds a disk in $V_j$ since $p_j^1$ did.) To extend $\QQ_j$, observe that since $\gamma_j$ traverses every seam of $\QQ_j$, $z_j$ is in either a rectangle or a hexagon of $Q_j | \gamma_j$ (see Figure \ref{figure:recorhex}). The component of $\gamma_j \cap Q_j$ containing $\sigma \cap Q_j$ connects two boundary components of $Q_j$, say $q_j^1$ and $q_j^2$. To extend $\QQ_j$, let $q^*_j$ be the unique curve in $Q_j$ which is parallel to $q_j^1$ in $Q_j$, but not in $Q_j \setminus n(z_j)$. (Notice again that $q_j^*$ bounds a disk in $W_j$ since $q_j^1$ did.)  See Figure \ref{figure:connectsum}, replacing $p$'s with $q$'s.  

Now, let ${\PP} = {\PP}_1 \cup {\PP}_2  \cup p^*_1 \cup p^*_2 \cup (\delta_1=\delta_2)$, and ${\QQ} = {\QQ}_1 \cup {\QQ}_2  \cup q^*_1 \cup q^*_2 \cup(\delta_1=\delta_2)$.

\begin{figure}[tb]
\begin{center}
\includegraphics[width=4in]{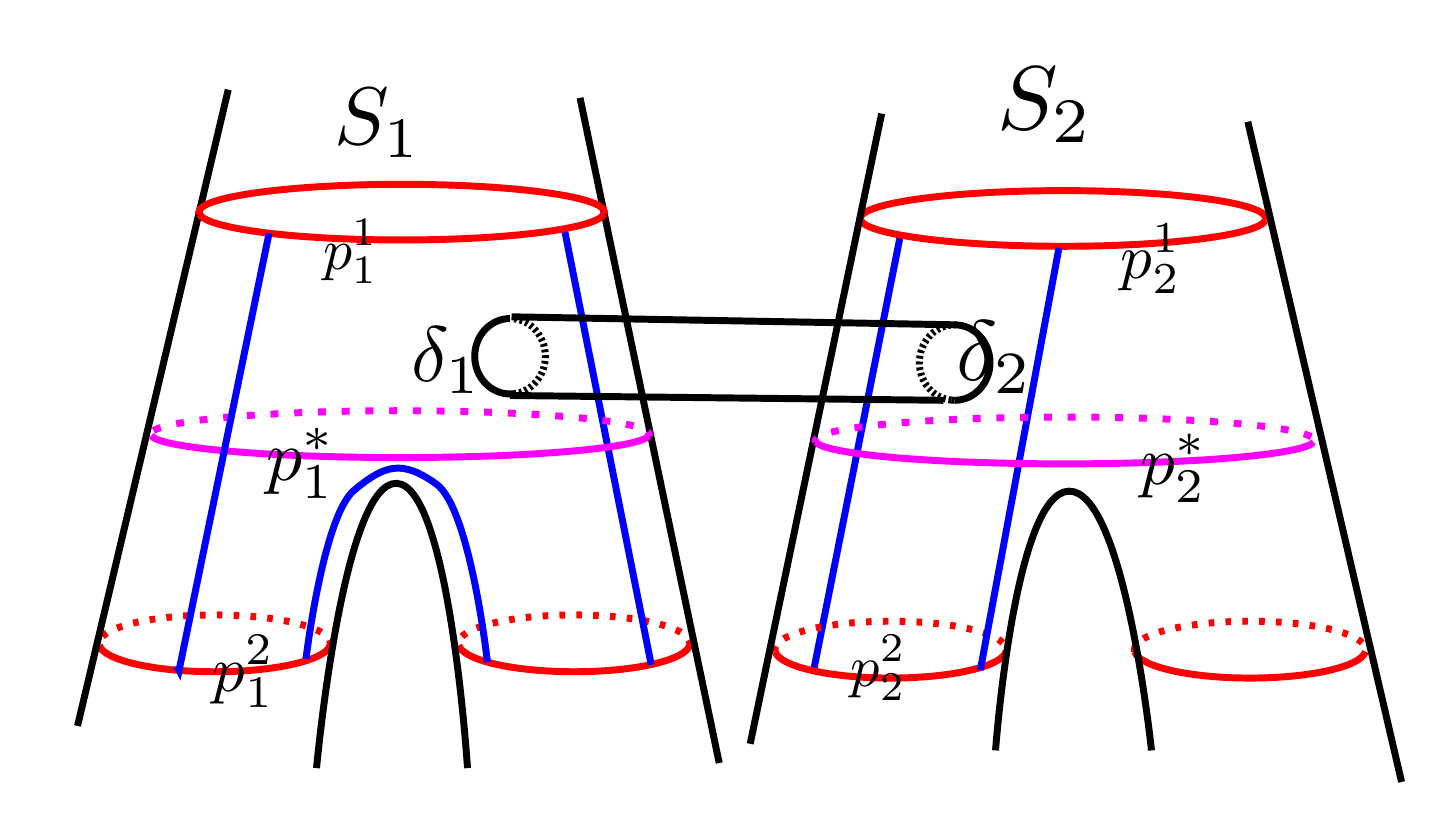}
\caption{Carefully connect summing.}
\label{figure:connectsum}
\end{center}
\end{figure}

\begin{figure}[tb]
\begin{center}
\includegraphics[width=2in]{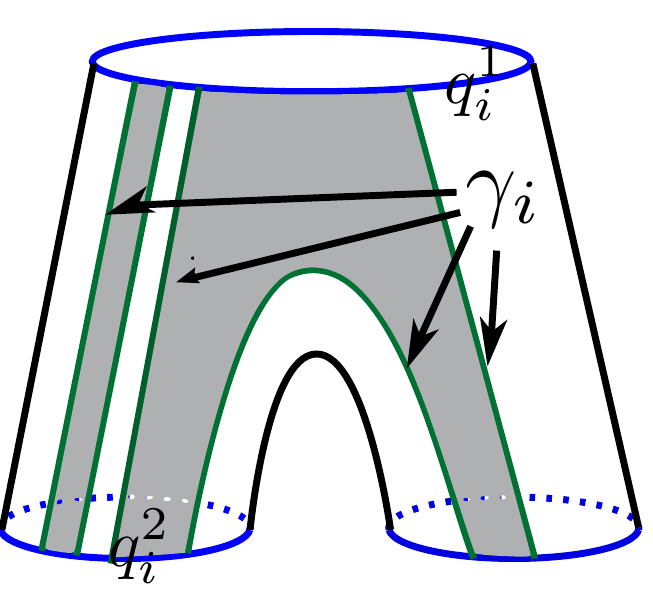}
\caption{$z_i$ is either in a rectangular or hexagonal region.}
\label{figure:recorhex}
\end{center}
\end{figure}

\begin{figure}[tb]
\begin{center}
\includegraphics[width=3in]{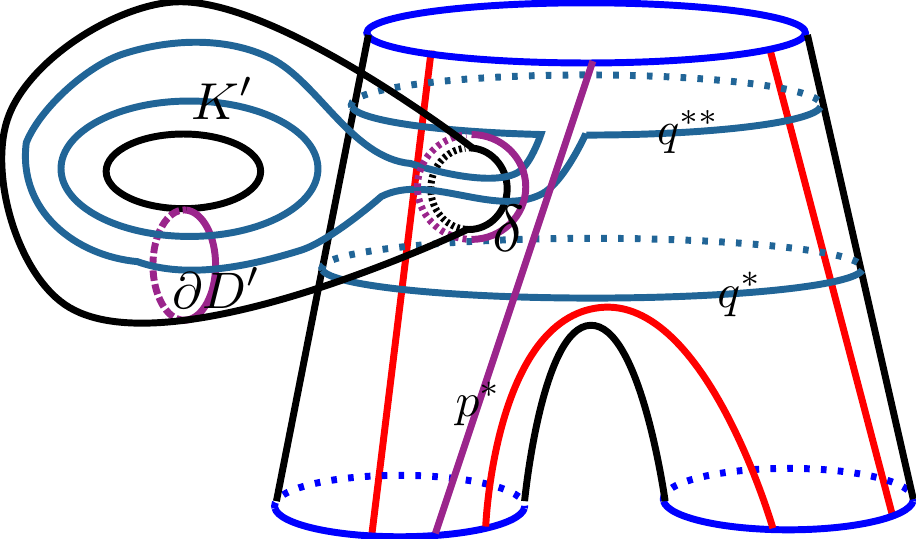}
\caption{The appended pants decompositions of $V$ and $W$.}
\label{figure:otherappendedpants}
\end{center}
\end{figure}

Since $z_j$ was chosen to be adjacent to $\sigma$, there is a path $v_j$ connecting $z_j$ to $\sigma$ whose interior does not intersect $\gamma_j$. To build $\gamma$, connect sum $\gamma_1$ to $\gamma_2$ along $v_1 \cup v_2$. Since the only changes to $\gamma_1 \cup \gamma_2$ happen inside the pairs of pants containing $\delta$, there are no bigons formed by $\gamma$ and $\PP \cup \QQ$. Thus, by the Bigon Criterion (see \cite{FarMarPMCG}), $\gamma$ intersects $\PP \cup \QQ$ minimally in its isotopy class. And we can verify directly that $\gamma$ traverses all the seams of $\PP$ and $\QQ$.

If there are more than two summands, we start as above with $M_1$ and $M_2$.  We then repeat the process using $M_1 \# M_2$ as the first manifold and $M_3$ as the second. Continue in this manner until $M = (S,V,W)$ is recovered, along with pants decompositions ${\PP}$ and ${\QQ}$, and a curve $\gamma$, which is of full type with respect to $\PP$ and $\QQ$.

\end{proof}

\section{Proof of Theorem \ref{theorem:Main Theorem}}
\label{sec:proof}
\begin{proof}
Let $(S, V, W)$ be a Heegaard splitting of a given closed $3$-manifold $M$, and let $n > 0$. By Lemma \ref{lem:curve}, there exist pants decompositions $\PP$ and $\QQ$ of $V$ and $W$, respectively, and a curve $\gamma$ which is of full type with respect to $\PP$ and $\QQ$.

Now we stabilize the splitting of $M$ once in a manner that allows us to control an extended pants decomposition of the stabilized Heegaard splitting.  The process of stabilizing and extending the pants decomposition is very similar to that of the connect summing done above.  

Again, let $z$ be a point in $S | (\PP \cup \QQ \cup \gamma)$  in a component adjacent to a segment $\sigma = \sigma(\gamma, p_1, p_2)$ of $\gamma$ connecting two different curves $p_1$ and $p_2$ of $\PP$. Then $z$ is in the intersection of two pairs of pants, say $P$ and $Q$, with $p_1$ and $p_2$ being two boundary components of $P$. Let $\Delta = \partial n(z)$. 

Since $z$ was chosen to be adjacent to $\sigma$, there is a path $v$ connecting $z$ to $\sigma$ whose interior does not intersect $\gamma$. 

Let $x, y$ be two points in the interior of $n(z)$, and join these by an arc $\alpha$ properly embedded in $W$, and isotopic to an arc $\beta \subset int(n(z))$. Then attach a 1-handle along $\alpha$, so that $\alpha \cup \beta=K$ is a core of the resulting handlebody, $V'$. Let $W' = \overline{ M - V' }$, and $S' = \partial V'$. Let $D'$ be the co-core of the 1-handle.

To extend $\PP$, let $p^*$ be the unique curve in $P$ which is parallel to $p_1$ in $S$, but not in $S'$. See Figure \ref{figure:otherappendedpants}. 

Since $\gamma$ traverses every seam of $\QQ$, $z$ is in either a rectangle or a hexagon of $Q | \gamma$ (see Figure \ref{figure:recorhex}). The component of $\gamma \cap Q$ containing $\sigma \cap Q$ connects two boundary components of $Q$, say $q_1$ and $q_2$. Let $w$ be a path in $P \cap Q$ connecting $z$ to $q_1$ whose interior does not intersect $\gamma$. 

To extend $\QQ$, first let $q^*$ be the unique curve in $Q$ which is parallel to $q_1$ in $S$, but not in $S'$. Let $K'$ be a curve on $S'$ parallel to $K$, bounding a disk in $W'$. Then, let $w'$ be an extension of the path $w$ into the stabilized part of $S'$ so that $w'$ meets $K'$, and let $q^{**}$ be the band sum of $q_1$ and $K'$ along $w'$ (see Figure \ref{figure:otherappendedpants}). 

Let ${\PP'} = {\PP} \cup \Delta \cup p^*$, and ${\QQ'} = {\QQ} \cup q^* \cup q^{**} \cup K'$. (Note that $\PP'$ is a pants decomposition for $V' \setminus n(K)$, not for $V'$.)

We now modify the curve $\gamma$ in $S$ constructed above to a new curve $a$ in $S'$ with three important properties:
\begin{itemize}
\item $a$ traverses every seam of $\PP'$.
\item $a$ traverses every seam of $\QQ'$.
\item $a$ bounds a disk in $V'$.
\end{itemize}

To do so, we begin with $\gamma$. Notice that $\gamma$ must traverse the seams (without loss of generality) $q^*q_1$ and $p^*p_1$. Now, take a product neighborhood of $v$, as above. The boundary of this neighborhood is a rectangle, made up of $v\times\{0\}$, $v\times\{1\}$, a segment of $\sigma$($\gamma$, $v\times\{0\}$ $v\times\{1\}$)=$\sigma_1$, and a segment of $\sigma$($\Delta$, $v\times\{0\}$ , $v\times\{1\}$).  Call the arc $\gamma'$= $v\times\{0\}$ $\cup$ $v\times\{1\}$ $\cup$ $\gamma \setminus \sigma_1$.  Extend each endpoint of $\gamma'$ slightly past $\Delta$ into the stabilized portion of $S'$.  Now the arc $\gamma'$ traverses the seams $p_1\delta$, $p^*\delta$, $q^*q_2$, and $q^*q_3$.  Take the band connect sum of two copies of a $D'$ along $\gamma'$.  The boundary of the resulting disk is the desired $a$. See Figure \ref{figure:makinga}. The curve $a$ now traverses all the previous seams, as well as the remaining seams of $\QQ'$: $q_1q^{**}$, $q_1K'$, $q^*q^{**}$, $q^*K'$, and both seams between $q^{**}$ and $K'$. See Figure \ref{figure:allseams}.

\begin{figure}[tb]
\begin{center}
\includegraphics[width=4in]{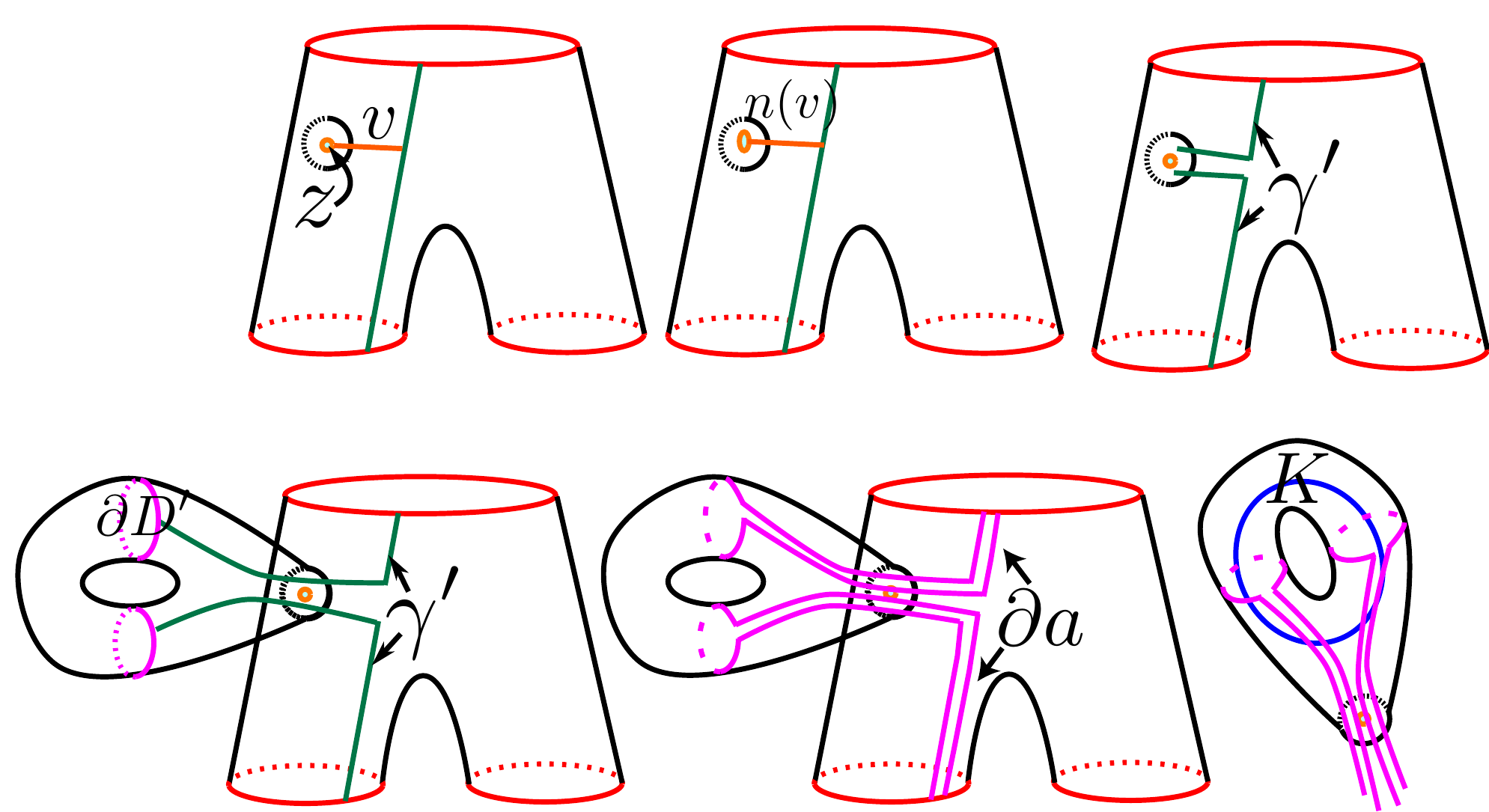}
\caption{Making the curve $\gamma'$ by modifying $\gamma$, and then band summing two copies of $D'$ along $\gamma'$ to make the meridian $a$.}
\label{figure:makinga}
\end{center}
\end{figure}

\begin{figure}[tb]
\begin{center}
\includegraphics[width=4in]{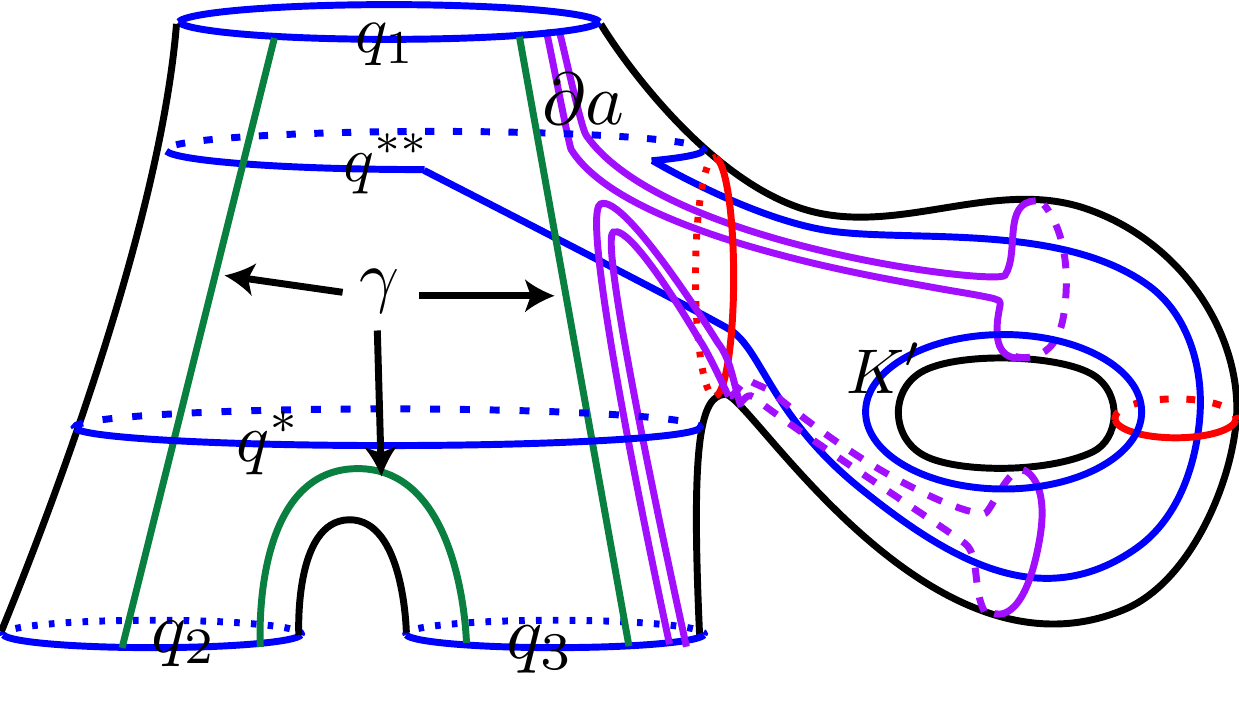}
\caption{All the seams of $\QQ$ are crossed by $a$.}
\label{figure:allseams}
\end{center}
\end{figure}

The proof now proceeds exactly as in \cite{MinMorSchHDK}. Choose two meridians $b, c \in \DD_{V'}$ so that together $b$ and $c$ fill $S'$, (see, e.g., Kobayashi \cite{KobPAHEORHS3}, Proof of Lemma 2.2). Let $\tau_a, \tau_b, \tau_c$ denote Dehn twists about $a, b,$ and $c$ respectively. Let $\Phi_0 = \tau_b \circ \tau_c^{-1}$. By Thurston's construction, \cite{ThuGDDS}, $\Phi_0$ is a pseudo-Anosov homeomorphism. Call its stable and unstable laminations $\lambda_0^{\pm}$. Then, define $\Phi_N = \tau_a^N \circ \Phi_0 \circ \tau_a^{-N}$. Because $a, b,$ and $c$ are all meridians, $\Phi_N$ extends over $V'$. Further, the stable and unstable laminations $\lambda_N^{\pm}$ of $\Phi_N$ are just $\tau_a^N(\lambda_0^{\pm})$. Now, $\lambda_0^{\pm}$ fills $S'$, so $a$ must intersect it, and thus the laminations $\lambda_N^{\pm}$ converge to $[a]$ in $\mathcal{S'}$ as $N \to \infty$. Hence, eventually both laminations are of full type with respect to $\PP$ and $\QQ$. Take $\Phi = \Phi_N$ for such a large $N$ that this occurs, and take $\lambda^{\pm} = \lambda_N^{\pm}$. Then $\Phi$ is a pseudo-Anosov homeomorphism of $S'$ which extends over $V'$, and whose stable and unstable laminations $\lambda^{\pm}$ are of full type with respect to $\PP$ and $\QQ$. 

Thus, by Lemma \ref{lem:lam}, $\lambda^{-} \not \in \overline{\DD_{V \setminus n(K)}}$ and $\lambda^+ \not \in \overline{\DD_W}$. And, by Lemma \ref{lem:dis}, $d(\Phi^n(\DD_{V \setminus n(K)}), \DD_W) \to \infty$ as $n \to \infty$, which proves Theorem \ref{theorem:Main Theorem}.

\end{proof}

\section{Corollaries and Discussion}
\label{section:corollaries}
\subsection{Surfaces}

Manifolds with high distance Heegaard splittings have been shown to possess many nice properties.  Being able to produce knots whose complements have high distance splittings means being able to produce knot complements with these properties. For example:  

\begin{cor}
Let $M=(S,V,W)$ be a Heegaard splitting of a closed, orientable $3$-manifold.  Fix $n>0$.   After one stabilization, it is possible to find a core of the new splitting whose exterior in $M$ has no closed essential surfaces of genus less than $n$.
\end{cor}

\begin{proof} 
Fix $n>0$.  Construct a knot $K$ as in Theorem \ref{theorem:Main Theorem}, which gives a splitting of $\overline{M-n(K)}$ of distance greater than $2n$. Then, by a theorem of Hartshorn \cite{HarHSHMHB}, $\overline{M-n(K)}$ has no closed incompressible surfaces of genus less than $n$. 

\end{proof}

This should be compared to a result of Kobayashi \cite{KobNisNSC3MHGGHS}. In Theorem 5.1, he shows that for $n>0$ if $M$ has a genus $h$ Heegaard splitting then there exists a knot $K$ whose exterior has no closed incompressible surfaces of genus less than $n$, but $K$ cannot sit as a core of a genus $h$ Heegaard splitting of $M$

A result of  Scharlemann and Tomova \cite{SchaTomAHGBD} shows that if a Heegaard splitting has sufficiently high distance in relation to its genus, it is the unique minimal genus Heegaard splitting.  This gives us:

\begin{cor}
Let $M=(S,V,W)$ be a genus $g$ Heegaard splitting of a closed, orientable $3$-manifold.  There exists a knot $K \subset M$ whose complement has Heegaard genus $g+1$. 
 \end{cor}

\begin{proof}
Start with a genus $g$ Heegaard splitting of $M$. Construct a knot $K$ as in Theorem \ref{theorem:Main Theorem} with a splitting of $\overline{M-n(K)}$ of distance greater than $2(g +1)$, and genus $g+1$. Then this will be the unique minimal genus Heegaard splitting of the knot complement.
\end{proof}

We say $K\subset M$ has a $(t,b)$-decomposition if $K$ can be isotoped to have exactly $b$ bridges with respect to a genus $t$ Heegaard splitting of $M$.  A $(t,0)$-decomposition requires $K$ to be a core of a genus $t$ Heegaard splitting.  Another result of Tomova \cite{TomDHSKC} yields:

\begin{cor}Let $M$ be a genus $g$, closed, orientable $3$-manifold. For any positive integers $t \geq g$ and $b$, there is a knot $K\subset M$ with tunnel number $t$ so that $K$ has no $(t,b)$-decomposition.  
\end{cor}

\begin{proof}
Following [Section 4, \cite{MinMorSchHDK}], choose a Heegaard splitting of genus $t$ and construct a knot $K$ as in Theorem \ref{theorem:Main Theorem} with splitting $(S, V, W)$ of genus $t+1$ and of distance $> 2t + 2b$.  

Suppose $K$ admitted a $(h, c)$-decomposition for some $h \leq t$ and $c \leq b$. Let $Q$ be the Heegaard surface of $S^3$ associated to this decomposition. As the genus of $Q$ is less than $t+1$, it cannot be isotopic to (a stabilization of) $S$. So Theorem 1.3 of \cite{TomDHSKC} tells us that $d(S) \leq 2 - \chi(Q \setminus n(K)) = 2h + 2c \leq 2t + 2b$, a contradiction.

Further, letting $c = 0$ tells us that $\overline{M \setminus n(K)}$ has no splitting of genus less than $t+1$, and thus $K$ has tunnel number $t$.
\end{proof}

\subsection{Coarse Geometry}
In this section, we shall assume that $S$ is a closed surface. If two handlebody disk sets are distinct, we prove that they are, in fact, coarsely distinct. That is, 

\begin{theorem}
\label{theorem:CoarselyDistinct}
If $\DD_V$, $\DD_W \subset \mathcal{C}_S$ are two different handlebody disk sets, and $g(S) \geq 2$, then for all $K \geq 0$, $\DD_W \not \subset n_K(\DD_V)$, where $n_K(\cdot)$ denotes a $K-neighborhood$. 
\end{theorem} 

\begin{proof} 
Two handlebody sets for the same surface determine a manifold with Heegaard splitting, $M = (S, V, W)$. Since $\DD_V \neq \DD_W$, $M \neq \#_k S^1 \times S^2$.  We consider separately the cases that $M = (S, V, W)$ is stabilized and the case that $M = (S, V, W)$ is not stabilized and in each case construct a meridian.

If $M = (S, V, W)$ is stabilized, destabilize the splitting once so that $M = (\hat{S},\hat{V},\hat{W})$.  Construct the pants decompositions $\PP$ and $\QQ$ of $V$ and $W$ as in Lemma \ref{lem:prettypants} as well as the curve $\gamma$ as in the proof of Theorem \ref{theorem:Main Theorem}.  Restabilize the splitting so that $M = (S, V, W)$ and construct the meridian $a$ , and extend the pants decompositions to $\PP'$ and $\QQ'$ as in the proof of Theorem \ref{theorem:Main Theorem}.  Note that $\partial a$ is of full type with respect to $\QQ'$, the pants decomposition of $W$.  

If $M = (S, V, W)$ is not stabilized then we use a slightly different construction.  As before, if the splitting is reducible, take a maximal collection of separating, reducing spheres $\{F_1, \dots, F_{l-1}\}$ for the splitting, so that $M = (S_1, V_1, W_1) \#_{F_1} (S_2, V_2, W_2) \#_{F_2} \dots \#_{F_{l-1}} (S_l, V_l, W_l)$, where the $(S_i, V_i, W_i)$ are irreducible splittings of prime $3$-manifolds for $1 \leq i \leq k$, and are $S^2 \times S^1$ for $k+1 \leq i \leq l$.   Note that since $(S, V, W)$ is not stabilized, $(S_i, V_i, W_i)$ is not a genus $1$ splitting of $S^3$ for all $i$.  

We will single out $(S_1, V_1, W_1)$, an irreducible splitting of a prime $3$-manifold, and construct a curve $\gamma_1$ on $S_1$ in the following way. Pick a pair of cut systems $\DD_1$ and $\EE_1$ for $V_1$ and $W_1$ respectively which intersect minimally. As in Lemma \ref{lem:prettypants}, construct pants decompositions $\PP_1$ and $\QQ_1$ for $V_1$ and $W_1$ respectively, of full type. Notice that, in particular, every seam of $\QQ_1$ is traversed by a curve of $\DD_1 \subsetneq \PP_1$. Now, if $g(S_1) \geq 2$, then let $D_1$ be a disk in $V_1$, disjoint from $\DD_1$, and let $a_1$ be the boundary curve resulting from banding two copies of $D_1$ together along an arc that wraps at least twice around each curve of $\DD_1$. (See Figure \ref{figure:gamma}).  If $g(S_1) = 1$, then let $a_1$ be the boundary of the unique disk in $V_1$.  Since $S_1$ is a lens space $\partial a_1$ intersects the single pants curve $Q=\QQ_1$ for $W_1$ at least twice with the same orientation and thus is of full type with respect to $\QQ_1$ of $W_1$.

\begin{figure}[tb]
\begin{center}
\includegraphics[width=4in]{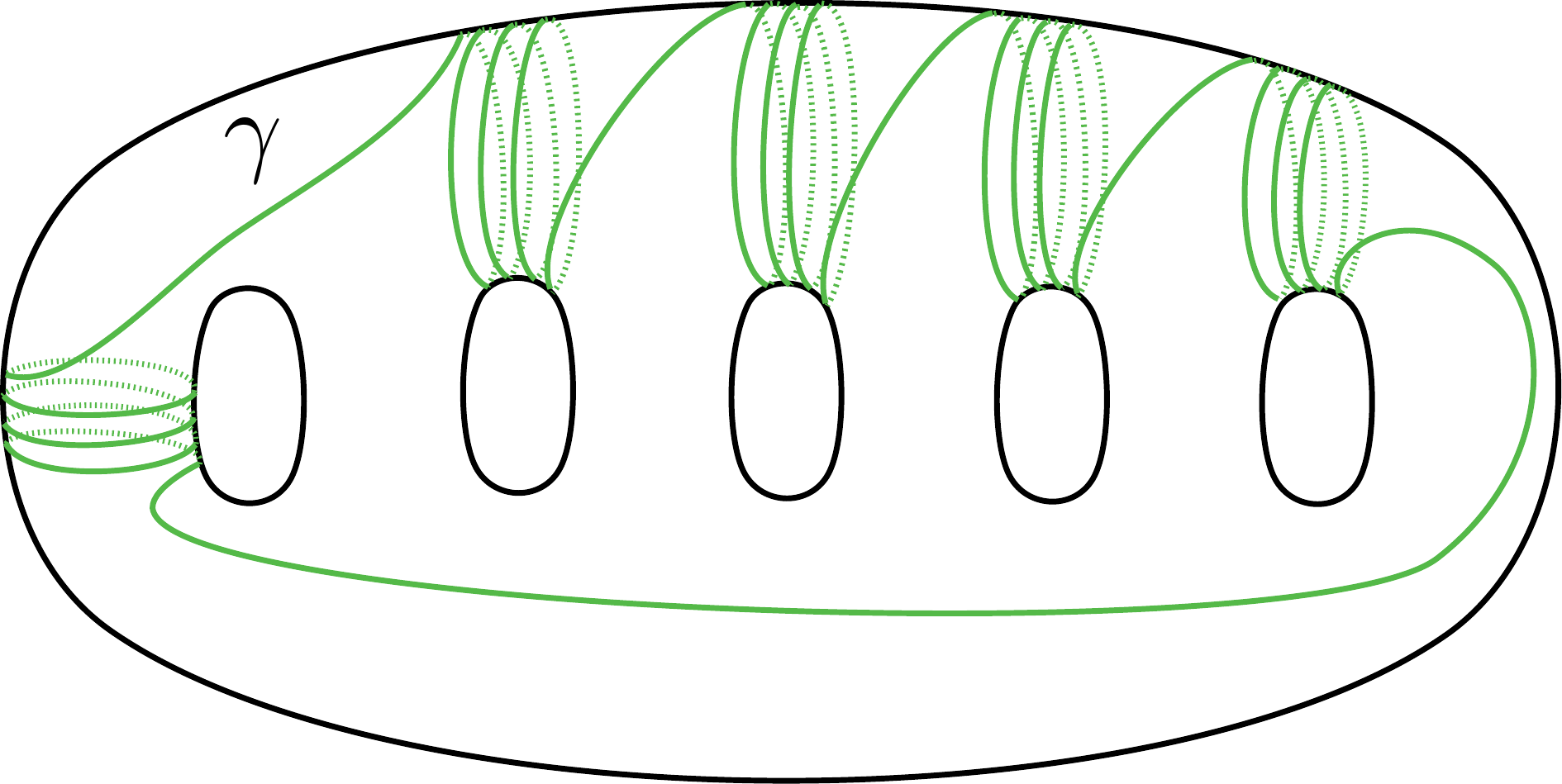}
\caption{The curve $\gamma$ is of full type with respect to $\QQ_1$.  To form $\gamma_1$ band sum two copies of a disk $D$ which intersects $\gamma$ once  along $\gamma$.}
\label{figure:gamma}
\end{center}
\end{figure}

For all $i > 1$, construct the curve $\gamma_i$, exactly as in the proof of Theorem \ref{theorem:Main Theorem}. Then reconstruct $S$, by connect summing the $S_i$, extend the pants decompositions of $W_i$ to the pants decomposition $\QQ$ of $W$, and construct the curve $\gamma$ by cutting and pasting together the $\gamma_i$ all exactly as in the proof of Theorem \ref{theorem:Main Theorem}. Then, $\gamma$ is of full type with respect to $\QQ$, and bounds a disk, $a$,  in $V$.

In either case, use $a$ to construct the map $\Phi$ in exactly the same way, and we have a pseudo-Anosov homeomorphism which extends to $V$, and whose stable and unstable laminations are of full type with respect to $\QQ$. 

Now, pick any disk $D^* \subset V$, and observe that in the curve complex, $d(\Phi^n(D^*), \DD_W) \to \infty$, which proves the theorem.
\end{proof} 
 
\begin{remark} The condition that $g(S) \geq 2$ is necessary, as the disk complex of any genus $1$ handlebody is a unique point, so they are all coarsely identical. 
\end{remark}

Gabai \cite{GabAFLCELS} has shown that $\partial \mathcal{C}_S$ is connected as long as $g(S) \geq 2$. Thus, recent theorems of Rafi and Schleimer tell us: 
	
\begin{theorem}[Rafi-Schleimer \cite{RafSchlCCCBR}] \label{theorem:bdddistance} 
For $g(S) \geq 2$, every quasi-isometry of $\mathcal{C}_S$ is bounded distance from a simplicial automorphism of $\mathcal{C}_S$.
\end{theorem} 

\begin{cor}[Rafi-Schleimer \cite{RafSchlCCCBR}] \label{cor:QI}
For $g(S) \geq 2$, $QI(\mathcal{C}_S)$ is isomorphic to $Aut(\mathcal{C}_S)$, the group of simplicial automorphisms.
\end{cor}

In other words, $\MCG(S) \cong \CMCG(S)$. Using Theorem \ref{theorem:CoarselyDistinct}, we prove an analogous result for the mapping class group of Heegaard splittings.

\begin{cor} For a Heegaard splitting $(S, V, W)$, of genus $\geq 2$, $\MCG(S, V, W) \cong \CMCG(S,V,W)$.

\end{cor}

\begin{proof}
There is an injective homomorphism from $\MCG(S, V, W)$ into $\CMCG(S, V, W)$ guaranteed by the restriction of the isomorphism from Corollary \ref{cor:QI} to mapping classes that fix $V$ and $W$. We prove that the image of this restriction is $\CMCG(S, V, W)$.

Let $f \in \CMCG(S, V, W)$. Then Theorem \ref{theorem:bdddistance} applies to tell us that $f$ is bounded distance from an isometry of $\mathcal{C}_S$, $g$.  So $f \sim g$, and $g$ also sends each handlebody set $\DD_V$ and $\DD_W$ to handlebody sets a bounded distance away. But by Theorem \ref{theorem:CoarselyDistinct}, $g$ must send $\DD_V$ and $\DD_W$ back onto themselves. Thus, $g \in \MCG(S, V, W)$. 
\end{proof}

%\bibliographystyle{hplain}   
%    \bibliography{RathbunReferences}

\end{document}